\documentclass[11pt,a4paper,reqno]{amsart}
\usepackage{amsmath,amssymb,amsthm,amsfonts, amscd, color, mathdots}
\usepackage{newtxtext,newtxmath}
\usepackage{mathrsfs}
\usepackage{verbatim}
\numberwithin{equation}{subsection}
\newtheorem{lemma}{Lemma}[section]

\newtheorem{theorem}{Theorem}[section]

\newtheorem{Remark}{Remark}[section]

\begin{document}

\title[Berndt-Arakawa formula]
  {On a certain identity for the cotangent finite Dirichlet series
  and its application to the Berndt--Arakawa formula
   } % 'Main title'
% 'Each author has his or her own set of coordinates.'
%%%%%%%%
%
%
%
%%%%%%%%%
\author{Masaaki Furusawa}
\address[Masaaki Furusawa]{
Department of Mathematics, Graduate School of Science,
              Osaka Metropolitan University,
         Sugimoto 3-3-138, Sumiyoshi-ku, Osaka 558-8585, Japan
}
\email{furusawa@omu.ac.jp}
\thanks{The research of the first author was supported in part by 
JSPS KAKENHI Grant Number
22K03235
and by Osaka Central  Advanced Mathematical
Institute (MEXT Promotion of Distinctive Joint Research Center Program JPMXP0723833165).}
%%%%%%%%
%
%
%
%%%%%%%%
\author{Tomo Narahara}
\address[Tomo Narahara]{
Department of Mathematics, Faculty of Science,
              Osaka City University,
         Sugimoto 3-3-138, Sumiyoshi-ku, Osaka 558-8585, Japan}
\email{a21sa016@st.osaka-cu.ac.jp}
%\thanks{}
%%%%%%%%%%%%%%%
%%%%%%
%%%%%%
%%%%%%
%
%
%
%%%%%%
\begin{abstract}
The cotangent zeta function is a very  interesting object,
which is related to partial zeta functions and 
Hecke $L$-functions of real quadratic fields.
Its special values at odd integers greater than $1$ are
explicitly evaluated by Berndt in the real quadratic unit case.
Later Arakawa generalized the formula
to the arbitrary real quadratic number case.

The purpose of this article is to give a novel
and surprisingly simple proof 
of the Berndt--Arakawa formula.
Namely we prove a certain asymptotic identity for two  finite cotangent
Dirichlet series, from which we deduce
the Berndt--Arakawa formula  immediately.
The authors believe that
the method employed here of evaluating  special values of
\emph{infinite} Dirichlet series  from an asymptotic identity for its 
\emph{finite} partial sum    is 
of interest in its own right.
\end{abstract}
%%%%%%%%%%%
%
%
%
%%%%%%%%%%%
\date{\today}
%%%%%%%%%%%%%%
%
%
%
%
%
%
%
%%%%%%%%%%%%%
\maketitle
%%%%%%%%%%%
%
%
%
%
%%%%%%
\section{Introduction}
For a real irrational algebraic number $\alpha$, the
Dirichlet series 
\begin{equation}\label{cot zeta}
\xi\left(s,\alpha\right)=
\sum_{n=1}^\infty\,\frac{\cot \pi n\alpha}{n^s}
\end{equation}
converges absolutely when $\mathrm{Re}\left(s\right)>1$,
thanks to the Thue--Siegel--Roth theorem in the Diophantine approximation theory, 
as shown in Arakawa~\cite[Lemma~2.1]{Arakawa2} by using \cite[Lemma~1]{Arakawa1}.
We call $\xi\left(s,\alpha\right)$ the cotangent zeta function.

Suppose that   $\alpha$ is a real quadratic number.
Then
$\xi\left(s,\alpha\right)$ is of substantial interest because of its 
relationship with partial zeta functions and Hecke $L$-functions of real quadratic fields,
which we refer to Arakawa~\cite{Arakawa2}
and Arakawa, Ibukiyama and Kaneko~\cite[Chapter 13]{AIK}.
Berndt~\cite[Theorem~5.2]{Berndt} proved a fascinating  explicit
formula for $\xi\left(s,\alpha\right)$
when $s$ is an odd integer greater than $1$ and
$\alpha$ is a unit.
The formula was  later generalized  to an arbitrary real quadratic number $\alpha$
by
Arakawa~\cite[(2.7)]{Arakawa2}
(see also \cite[Theorem~13.4]{AIK}).

The purpose of this article is to give a novel proof of the Berndt--Arakawa
formula.
Indeed we deduce it  from a certain asymptotic identity for the finite cotangent Dirichlet series.
All existing proofs rely on meromorphic continuations in some way or  other.
On the contrary, our proof is surprisingly simple, not appealing to meromorphic continuations at all.
To the best of the authors' knowledge,
this is the first example of evaluating special values of \emph{infinite} Dirichlet series
from an asymptotic identity for \emph{finite} Dirichlet series.
We think that our method  is of  interest in its own right.

The authors believe that our approach is applicable to evaluating 
explicitly
special values of some other Dirichlet series by 
replacing $S_m\left(k\right)$ in
\eqref{the series}, employed in this paper, by  appropriate polynomials.
%%%
%
%%%
\subsection{Notation}
$\mathbb N$ denotes
the set of positive integers, i.e.
$\mathbb N=\left\{n\in\mathbb Z\mid n\ge 1\right\}$.
For $x\in\mathbb R$, we denote by $\lfloor x\rfloor$ the greatest integer
less than or equal to  $x$ and by $\left\{x\right\}$ the fractional part of $x$,
i.e. $\left\{x\right\}=x-\lfloor x\rfloor$.

For $x\in\mathbb R$, let $e\left(x\right):=\exp\left(2\pi ix\right)$.
Note that $e\left(x\right)=e\left(\left\{x\right\}\right)$.

We denote by $B_n\left(x\right)$ the Bernoulli polynomial defined by
\begin{equation}\label{Bernoulli}
\frac{te^{tx}}{e^t-1}=\sum_{n=0}^\infty
\frac{B_n\left(x\right)}{n !}\,t^n.
\end{equation}
We note that $B_0\left(x\right)=1$ in particular.
%and by $s\left(d,c\right)$ the Dedekind sum defined by
%\begin{equation}\label{dedekind sum}
%s\left(d,c\right)=
%\begin{cases}
%\displaystyle{\sum_{j=1}^{c-1}\,B_1\left(\frac{j}{c}\right)B_1
%\left(\left\{\frac{jd}{c}\right\}\right)},
%&\text{when $c>1$};
%\\
%0,&\text{when $c=1$}
%\end{cases}
%\end{equation}
%for $c\in\mathbb N$ and $d\in\mathbb Z$,
%where we note that $B_1\left(x\right)=x-1\slash 2$.

%%%
%
%%%
\subsection{Main Theorem}

Now let us state our main theorem.
%%%
%
%%%
\begin{theorem}\label{main theorem}
Let $\alpha$ be an  irrational real number.
For $k,m\in\mathbb N$, 
let
\begin{equation}\label{def of finite cotangent}
\xi_k\left(2m-1,\alpha\right):=
\sum_{n=1}^k\, \frac{\cot\pi n\alpha}{n^{2m-1}}.
\end{equation}

For $V=\begin{pmatrix}a&b\\c&d\end{pmatrix}\in
\mathrm{SL}_2\left(\mathbb Z\right)$,
we define $\displaystyle{V\alpha:=\frac{a\alpha+b}{c\alpha +d}}$
and suppose that $\eta:=c\alpha+d>0$ with $c>0$.

Then for $m\ge 2$, we have
\begin{multline}\label{main identity2}
\xi_K\left(2m-1,\alpha\right)-
\eta^{2m-2}\,\xi_k\left(2m-1,V\alpha\right)+
\frac{\eta^{2m-2}}{\pi k^{2m-1}}\cdot
\frac{1}{1-\left\{\frac{K-a k}{c}\right\}-c^{-1}
\left\{\frac{k}{\eta}\right\}}
\\
=
\left(-1\right)^{m-1}\left(2\pi\right)^{2m-1}
\sum_{\ell=0}^{2m}\,
\sum_{j\,\mathrm{mod}\,c}
\frac{B_\ell\left(x_j\right)B_{2m-\ell}\left(y_j\right)}{
\ell !\left(2m-\ell\right)!}\cdot
\eta^{2m-\ell-1}+O\left(k^{-1}\right).
\end{multline}
Here 
\[
K=\lfloor \frac{k}{\eta}\rfloor
\quad\text{and}\quad
x_j=1-\left\{\frac{dj}{c}\right\},
y_j=\left\{\frac{j}{c}\right\}
\quad\text{for each $j\,\mathrm{mod}\,c$}.
\]
\end{theorem}
\begin{Remark}
In this theorem, we are dealing with the finite series $\xi_k\left(2m-1,\alpha\right)$
and hence the convergence is not an issue.
Thus we do not impose the condition that $\alpha$ is algebraic.
\end{Remark}
%%%
%
%%%
%%%
%
%%%

The deduction of the Berndt--Arakawa formula
from \eqref{main identity2} is immediate.
%%%
%
%%%
\begin{theorem}[Berndt~\cite{Berndt}, Arakawa~\cite{Arakawa2}]
\label{ba formula}
Let $F$ be a real quadratic field and $\alpha\in F\setminus \mathbb Q$.
Take a totally positive unit $\eta\in F$ such that
$\displaystyle{
\begin{pmatrix}\eta\alpha\\ \eta\end{pmatrix}
=V\begin{pmatrix}\alpha\\ 1\end{pmatrix}
}$
for some $V=\begin{pmatrix}a&b\\c&d\end{pmatrix}
\in\mathrm{SL}_2\left(\mathbb Z
\right)$ with $c>0$.

Then for an integer $m\ge 2$, we have
\begin{equation}\label{e: ba formula}
\xi\left(2m-1,\alpha\right)=
\frac{\left(-1\right)^{m-1}\left(2\pi\right)^{2m-1}}{1-\eta^{2m-2}}
\sum_{\ell=0}^{2m}\,
\sum_{j\,\mathrm{mod}\,c}
\frac{B_\ell\left(x_j\right)B_{2m-\ell}\left(y_j\right)}{
\ell !\left(2m-\ell\right)!}\cdot
\eta^{2m-\ell-1}.
\end{equation}
\end{theorem}
%%%
%
%%%
\begin{proof}
For the existence of such $\eta$, we refer to 
Arakawa, Ibukiyama and Kaneko~\cite[page 216]{AIK}.

In \eqref{main identity2}, we note that
\[
1-\left\{\frac{K-a k}{c}\right\}-c^{-1}
\left\{\frac{k}{\eta}\right\}\ge 1-\left(1-c^{-1}\right)-c^{-1}
\left\{\frac{k}{\eta}\right\}=c^{-1}\left(1-\left\{\frac{k}{\eta}\right\}\right)
\ge c^{-1} \langle\langle\frac{k}{\eta}\rangle\rangle
\]
where $\langle\langle x\rangle\rangle$ denotes the distance to the nearest integer
for $x\in\mathbb R$.
By Thue--Siegel--Roth theorem on Diophantine approximation
in transcendental number theory,  for any $\varepsilon>0$,
there exists a constant $C\left(\varepsilon\right)>0$ such that
\[
\langle\langle\frac{k}{\eta}\rangle\rangle>C\left(\varepsilon\right) k^{-1-\varepsilon}
\]
for any $k\in\mathbb N$ (cf. \cite[page 213]{AIK}).
Hence
\[
\left|
\frac{\eta^{2m-2}}{\pi k^{2m-1}}\cdot
\frac{1}{1-\left\{\frac{K-a k}{c}\right\}-c^{-1}
\left\{\frac{k}{\eta}\right\}}
\right|
< \frac{\eta^{2m-2}}{\pi k^{2m-1}}
\cdot\frac{c k^{1+\varepsilon}}{C\left(\varepsilon\right)}=
\frac{\eta^{2m-2}c}{\pi C\left(\varepsilon\right)}\, k^{2-2m+\varepsilon}.
\]

Thus by letting $k\to\infty$ and noting that $V\alpha=\alpha$, 
we obtain \eqref{e: ba formula}
from \eqref{main identity2}.
\end{proof}
%%%
%
%
%
%%%
\begin{Remark}
We remark that by taking $V=\begin{pmatrix}0&-1\\1&0\end{pmatrix}$,
we also obtain the following  functional equation, which Lerch stated
in \cite{Lerch} without a proof, by the same argument from Theorem~\ref{main theorem}:

For an algebraic irrational real number $\alpha$ and an integer $m\ge 2$, 
we have
\begin{multline}\label{lerch}
\xi\left(2m-1,\alpha\right)-\alpha^{2m-2}\xi\left(2m-1,\alpha^{-1}\right)
\\=
\left(-1\right)^{m-1}\left(2\pi\right)^{2m-1}
\sum_{\ell=0}^{2m}
\frac{B_\ell B_{2m-\ell}}{\ell!\,\left(2m-\ell\right)!}\cdot \alpha^{2m-\ell-1}
\end{multline}
where $B_j$ denotes the $j$-th Bernoulli number.

We refer to Berndt~\cite[Section~6]{Berndt}
and the references therein
for the relation of   \eqref{lerch} to the first letter from Ramanujan to Hardy,
dated January 16, 1913, 
and to  Ramanujan's formula for $\zeta\left(2r-1\right)$, $r>1$.
\end{Remark}
%%%%%%%%%%%%%%%%%%%%%%
%
%
%
%
%%%%%%%%%%%%%%%%%%%%%%%
\section{Proof of Theorem~\ref{main theorem}}
%%%%%%%
%
%
%
%%%%%%%
\subsection{Some preliminaries and lemmas}
Before proceeding to the proof of Theorem~\ref{main theorem},
we recall some preliminaries and note some lemmas here.
%%%%%
%
%%%%%
\subsubsection{Lemma on Bernoulli polynomials}
\begin{lemma}\label{lemma1}
For $n,q\in\mathbb N$ and $x\in\mathbb R$, 
we define $A_{n,q}\left(x\right)$ by
\[
A_{n,q}\left(x\right):=\sum_{0<\left|u\right|\le n}
e\left(ux\right)u^{-q}
\quad\text{where $e\left(x\right)=\exp\left(2\pi i x\right)$}.
\]

Then we have
\begin{equation}\label{e:1-2}
A_{n,q}\left(x\right)=
\begin{cases}0,&\text{when $q=1$ and $x\in\mathbb Z$};
\\
\displaystyle{-B_q\left(\left\{x\right\}\right)\frac{\left(2\pi i\right)^q}{q!}
+O\left(\frac{1}{n}\right)},&\text{otherwise}.
\end{cases}
\end{equation}
\begin{proof}
When $q=1$ and $x\in\mathbb Z$, 
$A_{n,1}\left(x\right)=\sum_{0<\left|u\right|\le n}
u^{-1}=0$.

Suppose otherwise. We recall   (e.g. \cite[Theorem~4.11]{AIK}) :
\begin{equation}\label{bernoulli polynomial}
B_q\left(\left\{x\right\}\right)=-
\frac{q!}{\left(2\pi i\right)^q}\,\lim_{m\to\infty}A_{m,q}\left(x\right).
\end{equation}
Hence
\begin{align*}
A_{n,q}\left(x\right)&=\lim_{m\to\infty}A_{m,q}\left(x\right)
-\sum_{u=n+1}^\infty \left\{
e\left(ux\right)+\left(-1\right)^qe\left(-ux\right)\right\}u^{-q}
\\
&=-B_q\left(\left\{x\right\}\right)\frac{\left(2\pi i\right)^q}{q!}
-\sum_{u=n+1}^\infty \left\{
e\left(ux\right)+\left(-1\right)^qe\left(-ux\right)\right\}u^{-q}.
\end{align*}
When $q\ge 2$, we have
\[
\left|
\,\sum_{u=n+1}^\infty \left\{
e\left(ux\right)+\left(-1\right)^qe\left(-ux\right)\right\}u^{-q}
\right|
\le 2\sum_{u=n+1}^\infty\,\frac{1}{u\left(u-1\right)}=\frac{2}{n}.
\]
Suppose that  $q=1$. Then 
\begin{equation}\label{remainder}
\sum_{u=n+1}^\infty
\left\{e\left(ux\right)-e\left(-ux\right)\right\}u^{-1}
=2i\sum_{u=n+1}^\infty \frac{\sin \left(2\pi ux\right)}{u}.
\end{equation}
Put $A_r=\sum_{u=1}^r\sin\left(2\pi ux\right)$ for $r\in\mathbb N$.
Then by the well-known summation formula, we have
\[
A_r=\frac{\sin\left(\pi rx\right)\cdot
\sin \left\{\pi \left(r+1\right)x\right\}}{ \sin \pi x}
\]
since $\sin \left(\pi x\right)\ne 0$ from the assumption.
Hence $A_r=O\left(1\right)$.
By the Abel transformation, we have
\[
\sum_{u=n+1}^r\frac{\sin\left(2\pi ux\right)}{u}
=\frac{A_r}{r}-\frac{A_n}{n+1}+
\sum_{u=n+1}^{r-1}A_u\left(\frac{1}{u}-\frac{1}{u+1}\right)
\]
and hence
\[
\sum_{u=n+1}^\infty \frac{\sin \left(2\pi ux\right)}{u}
=O\left(\frac{1}{n}\right)+O\left(\sum_{u=n+1}^\infty \left(
\frac{1}{u}-\frac{1}{u+1}\right)\right)=O\left(\frac{1}{n}\right).
\]
\end{proof}
\end{lemma}
%%%%%%%%%%%%%%%%%%%%%%
%
%
%
%
%%%%%%%%%%%%%%%%%%%%%%
\subsection{Proof of Theorem~\ref{main theorem}}
Now let us prove our main theorem,
Theorem~\ref{main theorem}.

First, for $k,m\in \mathbb N$, 
we define a sum $S_{m}\left(k\right)$ by
\begin{equation}\label{the series}
S_{m}\left(k\right):=-
\sum_{\substack{ 0<\left|u\right|\le k,\,
0<\left|v\right|\le K
\\
c\lvert\left(u-dv\right)
}}\,
c \cdot f_m\left(u,v\right)
%\right)
\end{equation}
where
\begin{equation}\label{e:1-0}
f_m\left(u,v\right)=\frac{u^{2m-1}-\eta^{2m-1}v^{2m-1}}{
u^{2m-1}v^{2m-1}\left(u-\eta v\right)}=
\sum_{\ell=1}^{2m-1}\eta^{2m-1-\ell}u^{\ell-2m}v^{-\ell}.
\end{equation}

Then  Theorem~\ref{main theorem} is proved 
by deforming $S_{m}\left(k\right)$
in two different ways.
%%%%%%
%
%
%
%%%%%%
\subsubsection{The first deformation}
For $u,v\in\mathbb Z$,
we recall that 
\begin{equation}\label{character}
\sum_{j\,\mathrm{mod}\,c}
e\left(\frac{u-dv}{c}\, j\right)
=\begin{cases}c,&\text{when $c\lvert\left(u-dv\right)$};
\\
0,&\text{otherwise}.
\end{cases}
\end{equation}

Hence we may write \eqref{the series} as
\begin{equation}\label{e:1-1}
S_{m}\left(k\right)=-
\sum_{\ell=1}^{2m-1}\eta^{2m-1-\ell}
\sum_{j\,\mathrm{mod}\,c}
A_{k,2m-\ell}\left(\frac{j}{c}\right)\,
A_{K,\ell}\left(\frac{-dj}{c}\right).
\end{equation}

When $m\ge 2$, by Lemma~\ref{lemma1} and \eqref{e:1-1}, we have
\begin{equation}\label{e: first m=2}
S_m\left(k\right)=
\left(-1\right)^{m-1}\left(2\pi\right)^{2m}
\sum_{\ell=1}^{2m-1}
\sum_{j\,\mathrm{mod}\,c}
\frac{B_\ell\left(x_j\right)B_{2m-\ell}\left(y_j\right)}{\ell !\left(2m-\ell\right)!}
\eta^{2m-1-\ell}+O\left(k^{-1}\right).
\end{equation}
%%%%%%%%%
%
%
%
%%%%%%%%%
\subsubsection{The second deformation}
%%%
%
%%%
%%%
%
%%%
%\subsection{Proof of Theorem~\ref{main theorem}}
%%%
%
%%%
There is another way to deform $S_m\left(k\right)$.

Since
\[
f_m\left(u,v\right)=\frac{-1}{v^{2m-1}\left(\eta v-u\right)}
+\frac{\eta^{2m-2}}{u^{2m-1}\left(v-\eta^{-1}u\right)},
\]
we have  $S_m\left(k\right)=S_m^{(1)}\left(k\right)+
S_m^{(2)}\left(k\right)$, where
\[
S_m^{(1)}\left(k\right)=\sum_{\substack{ 0<\left|u\right|\le k\\
0<\left|v\right|\le K
\\
c\lvert\left(u-dv\right)
}}
\frac{c}{v^{2m-1}\left(\eta v-u\right)}
\quad\text{and}
\quad
S_m^{(2)}\left(k\right)=\sum_{\substack{ 0<\left|u\right|\le k
\\
0<\left|v\right|\le K
\\
c\lvert\left(u-dv\right)
}}
\frac{-c\eta^{2m-2}}{u^{2m-1}\left(v-\eta^{-1}u\right)}.
\]
Then we have
\begin{align*}
&S_m^{(1)}\left(k\right)=\sum_{0<\left|v\right|\le K}
\frac{1}{v^{2m-1}}
\sum_{\substack{\left|u\right|\le k\\ c\lvert \left(u-dv\right)}}
\frac{1}{\frac{\eta v-u}{c}}
-\frac{c}{\eta}\sum_{\substack{0<\left|v\right|\le K\\ c\lvert dv}}
\frac{1}{v^{2m}}
\\
=&\sum_{0<\left|v\right|\le K}
\frac{1}{v^{2m-1}}
\sum_{\substack{\left|u\right|\le k\\ c\lvert \left(u-dv\right)}}
\frac{1}{\alpha v-\frac{u-dv}{c}}
-\frac{1}{\eta}\sum_{0<\left|v\right|\le K}
\frac{1}{v^{2m}}
\sum_{j\,\mathrm{mod}\,c} e\left(\frac{-dv}{c}j\right)
\\
=&\sum_{0<\left|v\right|\le K}
\frac{1}{v^{2m-1}}
\sum_{-\frac{k+dv}{c}\le w\le \frac{k-dv}{c}}
\frac{1}{\alpha v-w}
-\frac{1}{\eta}
\sum_{j\,\mathrm{mod}\,c}
\sum_{0<\left|v\right|\le K}
e\left(\frac{-dj}{c}v\right)
v^{-2m}.
\end{align*}
Hence
\begin{equation}\label{s_m^1}
S_m^{(1)}\left(k\right)=2\sum_{v=1}^{K}
\frac{1}{v^{2m-1}}
\sum_{-\frac{k+dv}{c}\le w\le \frac{k-dv}{c}}
\frac{1}{\alpha v-w}
-\frac{1}{\eta}\sum_{j\,\mathrm{mod}\,c} A_{K,2m}\left(
\frac{-dj}{c}\right).
\end{equation}

Similarly we have
\[
S_m^{(2)}\left(k\right)=
\sum_{0<\left|u\right|\le k}
\frac{-\eta^{2m-2}}{u^{2m-1}}
\sum_{\substack{\left|v\right|\le K\\ c\lvert\left(u-dv\right)}}
\frac{1}{c^{-1}\left(v-\eta^{-1}u\right)}
-\eta^{2m-1}
\sum_{j\,\mathrm{mod}\,c}A_{k,2m}\left(\frac{j}{c}\right).
\]
Since $u-dv=\left(ad-bc\right)u-dv=-bcu-d\left(v-au\right)$
and $c$, $d$ are relatively prime, we have
\[
c\lvert\left(u-dv\right)\Longleftrightarrow
c\lvert \left(v-au\right).
\]
When we put $w=c^{-1}\left(au-v\right)$, we have
\[
c^{-1}\left(v-\eta^{-1}u\right)=
c^{-1}\left(-cw+au-\eta^{-1}u\right)=
\frac{a\alpha+b}{c\alpha+d}\,u-w=V\alpha u-w.
\]
Hence
\begin{equation}\label{s_m^2}
S_m^{(2)}\left(k\right)=
-2\eta^{2m-2}\sum_{u=1}^{k}
\frac{1}{u^{2m-1}}
\sum_{-\frac{K-au}{c}\le w\le \frac{K+au}{c}}
\frac{1}{V\alpha u-w}
-\eta^{2m-1}
\sum_{j\,\mathrm{mod}\,c}A_{k,2m}\left(\frac{j}{c}\right).
\end{equation}
%%%%%%
%
%%%%%%
Thus we have
$S_m\left(k\right)=2T_m^{(1)}\left(k\right)-2\eta^{2m-2}T_m^{(2)}\left(k\right)
+U_m\left(k\right)$ where
\begin{align}
T_m^{(1)}\left(k\right)&=\sum_{v=1}^{K}
\frac{1}{v^{2m-1}}
\sum_{-\frac{k+dv}{c}\le w\le \frac{k-dv}{c}}
\frac{1}{\alpha v-w},
\label{t_m^1}
\\
T_m^{(2)}\left(k\right)&=\sum_{u=1}^{k}
\frac{1}{u^{2m-1}}
\sum_{-\frac{K-au}{c}\le w\le \frac{K+au}{c}}
\frac{1}{V\alpha u-w},
\label{t_m^2}
\\
U_m\left(k\right)&=-\frac{1}{\eta}\sum_{j\,\mathrm{mod}\,c} A_{K,2m}\left(
\frac{-dj}{c}\right)
-\eta^{2m-1}
\sum_{j\,\mathrm{mod}\,c}A_{k,2m}\left(\frac{j}{c}\right).
\label{u}
\end{align}
Here by Lemma~\ref{lemma1}, we have
\begin{equation}\label{u-0}
U_m\left(k\right)=
\left(-1\right)^m\left(2\pi\right)^{2m}
\sum_{\ell=0,\,2m}\,
\sum_{j\,\mathrm{mod}\,c}
\frac{B_\ell\left(x_j\right)B_{2m-\ell}\left(y_j\right)}{\ell !\left(2m-\ell\right)!}
\eta^{2m-1-\ell}+O\left(k^{-1}\right).
\end{equation}
%%%
%
%%%
\subsubsection{Estimation of the inner sums of  the double sums $T_m^{(1)}\left(k\right)$ and $T_m^{(2)}\left(k\right)$.}
First let us consider the inner sum of $T_m^{(1)}\left(k\right)$.
For $1\le v\le K$, we have
\begin{align*}
&\sum_{-\frac{k+dv}{c}\le w\le \frac{k-dv}{c}}
\,\frac{1}{\alpha v-w}
=\lim_{N\to \infty}\left( \sum_{-N\le w\le N}-\sum_{-N\le w< -\frac{k+dv}{c}}-\sum_{\frac{k-dv}{c}< w\le N}\right)\,\frac{1}{\alpha v-w}
 \\
&=\lim_{N\to \infty}\left( \sum_{-N\le w\le N}\frac{1}{\alpha v-w} -\sum_{\frac{k+dv}{c}< w\le N}
\frac{1}{w+\alpha v} +\sum_{\frac{k-dv}{c}< w\le N}\frac{1}{w-\alpha v} \right) \\
&=\pi\cot{\pi \alpha v} 
\\
&\quad+\lim_{N\to \infty}\left(-\sum_{\frac{k+dv}{c}+\lfloor\alpha v\rfloor< w\le N+\lfloor\alpha v\rfloor}\frac{1}{w+\{ \alpha v\}} +\sum_{\frac{k-dv}{c}-\lfloor\alpha v\rfloor< w\le N-\lfloor\alpha v\rfloor}\frac{1}{w-\{ \alpha v\}} \right) .
\end{align*}
%%%%%
%
%%%%%
Here when $N$ is sufficiently large, 
depending on the sign of $\alpha$, we note that 
\begin{align*}
&-\sum_{\frac{k+dv}{c}+\lfloor\alpha v\rfloor< w\le N+\lfloor\alpha v\rfloor}\frac{1}{w+\{ \alpha v\}} +\sum_{\frac{k-dv}{c}-\lfloor\alpha v\rfloor< w\le N-\lfloor\alpha v\rfloor}\frac{1}{w-\{ \alpha v\}} 
\\
=&\sum_{\frac{k-dv}{c}-\lfloor\alpha v\rfloor< w\le \frac{k+dv}{c}+\lfloor\alpha v\rfloor}\frac{1}{w-\{ \alpha v\}}
\\
&+\sum_{\frac{k+dv}{c}+\lfloor\alpha v\rfloor< w\le N - \lfloor\alpha v\rfloor}\left( \frac{1}{w-\{ \alpha v\}} -\frac{1}{w+\{ \alpha v\}}\right)-\sum_{N - \lfloor\alpha v\rfloor< w\le N + \lfloor\alpha v\rfloor}\frac{1}{w + \{ \alpha v\}}
\end{align*}
when $\alpha>0$, and, 
\begin{align*}
&-\sum_{\frac{k+dv}{c}+\lfloor\alpha v\rfloor< w\le N+\lfloor\alpha v\rfloor}\frac{1}{w+\{ \alpha v\}} +\sum_{\frac{k-dv}{c}-\lfloor\alpha v\rfloor< w\le N-\lfloor\alpha v\rfloor}\frac{1}{w-\{ \alpha v\}} 
\\
=&\sum_{\frac{k-dv}{c}-\lfloor\alpha v\rfloor< w\le \frac{k+dv}{c}+\lfloor\alpha v\rfloor}\frac{1}{w-\{ \alpha v\}}
\\
&+\sum_{\frac{k+dv}{c}+\lfloor\alpha v\rfloor< w\le N+ \lfloor\alpha v\rfloor}\left( \frac{1}{w-\{ \alpha v\}} -\frac{1}{w+\{ \alpha v\}}\right)+\sum_{N + \lfloor\alpha v\rfloor< w\le N - \lfloor\alpha v\rfloor}\frac{1}{w - \{ \alpha v\}}
\end{align*}
when $\alpha<0$.
Thus we have
\begin{multline}\label{sum}
\sum_{-\frac{k+dv}{c}\le w\le \frac{k-dv}{c}}
\,\frac{1}{\alpha v-w}
=
\pi\cot{\pi \alpha v} 
\\
+\sum_{\frac{k-dv}{c}-\lfloor\alpha v\rfloor< w\le \frac{k+dv}{c}+\lfloor\alpha v\rfloor}\frac{1}{w-\{ \alpha v\}}+\sum_{\frac{k+dv}{c}+\lfloor\alpha v\rfloor< w}\frac{2\{ \alpha v\}}{w^2-\{ \alpha v\}^2}.
\end{multline}
%%%%%
%
%%%%%

Let us estimate the second term of the right-hand side of \eqref{sum}.
Since
\[
\frac{k-dv}{c}-\lfloor\alpha v\rfloor-\{\alpha v\}=c^{-1}(k-\eta v)\ge c^{-1}(k-\eta K)=c^{-1}(k-\eta \lfloor \frac{k}{\eta}\rfloor)>0,
\]
each summand is positive and hence is bounded by
$\frac{1}{\frac{k-dv+1}{c}-\lfloor\alpha v\rfloor-\{\alpha v\}}=\frac{c}{k -\eta v +1}$.
Since the number of the terms in the sum is $O\left(v\right)$, the second term
of the right-hand side of \eqref{sum} is
$O\left(v\cdot\frac{c}{k -\eta v +1} \right)=
O\left(\frac{v}{k-\eta v+1}\right)$.

As for the third term of the right-hand side of \eqref{sum}, we have
\[
\sum_{\frac{k+dv}{c}+\lfloor\alpha v\rfloor< w}\frac{2\{ \alpha v\}}{w^2-\{ \alpha v\}^2}
\le\sum_{\frac{k}{c}-1< w}\frac{2}{w^2-1}
=O(k^{-1}).
\]

Since $k^{-1}=O\left(\frac{v}{k-\eta v+1}\right)$, we have
\begin{equation}\label{t_m^1-4}
\sum_{-\frac{k+dv}{c}\le w\le \frac{k-dv}{c}}\frac{1}{\alpha v-w}
=\pi\cot{\pi \alpha v} +O\left( \frac{v}{k -\eta v +1} \right).
\end{equation}
%%%%%
%
%
%
%%%%%

Now let us consider the inner sum of $T_m^{(2)}\left(k\right)$.
For $1\le u \le k$, we write
\[
\sum_{-\frac{K-au}{c}\le w\le \frac{K+au}{c}}\frac{1}{V\alpha u-w}
=\sum_{-\frac{K-au}{c}-1\le w\le \frac{K+au}{c}}\frac{1}{V\alpha u-w} -\frac{1}{V\alpha u+\lfloor\frac{K-au}{c}\rfloor+1}.
\]
By a computation similar to that above, we have
\begin{multline*}
\sum_{-\frac{K-au}{c}-1\le w\le \frac{K+au}{c}}\frac{1}{V\alpha u-w} =\pi\cot{\pi V\alpha u} 
\\
-\sum_{\frac{K-au}{c}+\lfloor V\alpha u\rfloor+1< w\le \frac{K+au}{c}-\lfloor V\alpha u\rfloor}\frac{1}{w+\{ V\alpha u\}}+\sum_{\frac{K+au}{c}-\lfloor V\alpha u\rfloor< w}\frac{2\{ V\alpha u\}}{w^2-\{ V\alpha u\}^2} 
\end{multline*}
and hence
\[
\sum_{-\frac{K-au}{c}-1\le w\le \frac{K+au}{c}}\frac{1}{V\alpha u-w} 
=\pi\cot{\pi V\alpha u} +O\left( \frac{u}{K-\frac{u}{\eta}+1+c} \right).
\]
Thus
\begin{equation}\label{t_m^2-3}
\sum_{-\frac{K-au}{c}\le w\le \frac{K+au}{c}}\frac{1}{V\alpha u-w}=
\pi\cot{\pi V\alpha u} -\frac{1}{V\alpha u+\lfloor\frac{K-au}{c}\rfloor+1}+O\left( \frac{u}{K-\frac{u}{\eta}+1+c} \right).
\end{equation}
%%%
%
%%%
\subsubsection{Estimation of $T_m^{(1)}\left(k\right)$ and $T_m^{(2)}\left(k\right)$}
From  \eqref{t_m^1}, \eqref{t_m^2}, \eqref{t_m^1-4} and \eqref{t_m^2-3}, we have
\begin{equation}\label{t_m^1-0}
T_m^{(1)}\left(k\right)=\pi\,\xi_K\left(2m-1,\alpha\right) +O\left( \sum_{v=1}^{K}\frac{1}{v^{2m-2}(k-\eta v+1)} \right)
\end{equation}
and
\begin{multline}\label{t_m^2-0}
T_m^{(2)}\left(k\right)=\pi\,\xi_k\left(2m-1,V\alpha\right)
-\sum_{u=1}^{k}\frac{1}{u^{2m-1}\left(V\alpha u+\lfloor\frac{K-au}{c}\rfloor+1\right)}
\\ +O\left( \sum_{u=1}^{k}\frac{1}{v^{2m-2}(K-\frac{u}{\eta}+1+c)} \right) .
\end{multline}
 As for the second term of the right-hand side of \eqref{t_m^1-0}, we have
\begin{align}\label{second term}
\begin{split}
&\sum_{v=1}^{K}\frac{1}{v^{2m-2}(k-\eta v+1)}
=\sum_{v=1}^{K}\frac{1}{v^{2m-2}\left(\eta K-\eta v+1+\eta\{\frac{k}{\eta}\}\right)} \\
&\le \frac{1}{K^{2m-2}\left(1+\eta\{\frac{k}{\eta}\}\right)}+\frac{1}{\eta}\sum_{v=1}^{K-1}\frac{1}{v^2(K-v)} \\
&=\frac{1}{K^{2m-2}\left(1+\eta\{\frac{k}{\eta}\}\right)}+\frac{1}{\eta}\left(\frac{1}{K}\sum_{v=1}^{K-1}\frac{1}{v^2}+\frac{1}{K^2}\sum_{v=1}^{K-1}\left(\frac{1}{v}+\frac{1}{K-v}\right)\right)
=O(k^{-1}).
\end{split}
\end{align}
Similarly, the third term of the right-hand side of \eqref{t_m^2-0} is $O(k^{-1})$. 

As for the second term of the right-hand side of \eqref{t_m^2-0}, noting that
\[
V\alpha u+\lfloor \frac{K-au}{c}\rfloor +1
=\frac{K-\frac{u}{\eta}}{c}+1-\left\{\frac{K-au}{c}\right\}=
\frac{k-u}{c\eta}
+1-\left\{\frac{K-au}{c}\right\}-
c^{-1}\left\{\frac{k}{\eta}\right\},
\]
we have
\begin{align*}
&\sum_{u=1}^{k}\frac{1}{u^{2m-1}\left(V\alpha u+\lfloor\frac{K-au}{c}\rfloor+1\right)} \\
&=\frac{1}{k^{2m-1}}\frac{1}{1-\left\{\frac{K-au}{c}\right\}-c^{-1}\left\{\frac{k}{\eta}\right\}}
+\sum_{u=1}^{k-1}\frac{1}{u^{2m-1}\left(\frac{k-u}{c\eta}+1-\left\{\frac{K-au}{c}\right\}-c^{-1}\left\{\frac{k}{\eta}\right\}\right)}.
\end{align*}
Then since
\[
\sum_{u=1}^{k-1}\frac{1}{u^{2m-1}\left(\frac{k-u}{c\eta}+1-\left\{\frac{K-au}{c}\right\}-c^{-1}\left\{\frac{k}{\eta}\right\}\right)} \le c\eta \sum_{u=1}^{k-1}\frac{1}{u^2(k-u)}=O(k^{-1})
\]
as seen in the argument for \eqref{second term}, we have
\[
\sum_{u=1}^k
\frac{1}{u^{2m-1}}\cdot\frac{1}{V\alpha u+\lfloor \frac{K-au}{c}\rfloor +1}=
\frac{1}{k^{2m-1}}\cdot
\frac{1}{1-\left\{\frac{K-ak}{c}\right\}-
c^{-1}\left\{\frac{k}{\eta}\right\}}+O\left(k^{-1}\right).
\]
%%%
%
%%%
%\subsubsection{The case when $m\ge 2$}
Hence
\begin{align}
&T_m^{(1)}\left(k\right)=\pi\xi_K\left(2m-1,\alpha\right)+O\left(k^{-1}\right), \label{t_m^1-1} \\
&T_m^{(2)}\left(k\right)=\pi\xi_k\left(2m-1,V\alpha\right)-\frac{1}{k^{2m-1}}\cdot \frac{1}{1-\left\{\frac{K-ak}{c}\right\}-c^{-1}\left\{\frac{k}{\eta}\right\}}+O\left(k^{-1}\right). \label{t_m^2-1}
\end{align}

Thus from \eqref{t_m^1-1}, \eqref{t_m^2-1}
and \eqref{u-0}, we have
\begin{multline}\label{s_m(k)-2}
S_m\left(k\right)=
2\pi
\left\{\xi_K\left(2m-1,\alpha\right)-\eta^{2m-2}
\,\xi_k\left(2m-1,V\alpha\right)\right\}
\\
+\left(-1\right)^m\left(2\pi\right)^{2m}
\sum_{\ell=0,\,2m}\,
\sum_{j\,\mathrm{mod}\,c}
\frac{B_\ell\left(x_j\right)B_{2m-\ell}\left(y_j\right)}{\ell !\left(2m-\ell\right)!}
\eta^{2m-1-\ell}
\\
+\frac{2\eta^{2m-2}}{k^{2m-1}}\cdot
\frac{1}{1-\left\{\frac{K-ak}{c}\right\}-
c^{-1}\left\{\frac{k}{\eta}\right\}}
+O\left(k^{-1}\right).
\end{multline}
Hence by comparing \eqref{e: first m=2} with 
\eqref{s_m(k)-2}, we have
\begin{multline*}
\left(-1\right)^{m-1}\left(2\pi\right)^{2m}
\sum_{\ell=1}^{2m-1}
\sum_{j\,\mathrm{mod}\,c}
\frac{B_\ell\left(x_j\right)B_{2m-\ell}\left(y_j\right)}{\ell !\left(2m-\ell\right)!}
\eta^{2m-1-\ell}+O\left(k^{-1}\right)
\\
=
2\pi
\left\{\xi_K\left(2m-1,\alpha\right)-\eta^{2m-2}
\,\xi_k\left(2m-1,V\alpha\right)\right\}
\\
+\left(-1\right)^m\left(2\pi\right)^{2m}
\sum_{\ell=0,\,2m}\,
\sum_{j\,\mathrm{mod}\,c}
\frac{B_\ell\left(x_j\right)B_{2m-\ell}\left(y_j\right)}{\ell !\left(2m-\ell\right)!}
\eta^{2m-1-\ell}
\\
+\frac{2\eta^{2m-2}}{k^{2m-1}}\cdot
\frac{1}{1-\left\{\frac{K-ak}{c}\right\}-
c^{-1}\left\{\frac{k}{\eta}\right\}}
+O\left(k^{-1}\right).
\end{multline*}
Thus \eqref{main identity2} holds.
%%%
%
%%%%%%%%%%%%%%%%%%%%%%
%
%
%
%
%%%%%%%%%%%%%%%%%%%%%%
\subsection*{Acknowledgements}
The authors would like to express their gratitude to the anonymous referee for carefully reading 
the manuscript and pointing out some inaccuracies and oversights in an earlier version.
They would like to thank
Masanobu Kaneko, Masaki Kato, Nobushige Kurokawa
and Yoshinori Mizuno for encouragement.
The first author would like to dedicate this article to the memory of Steve Zucker (1949--2019).
%%%%%%%%%%%%%%%%
%
%
%
%%%%%%%%%%%%%%%%%
\subsection*{Data Availability}
Data sharing is not applicable to this article as no datasets were generated or analyzed for this work. 
%%%%%%%%%%%%%%%%%
\section*{Declaration}
The authors have no conflicts of interest, financial or non-financial, to declare which are relevant to the content of this article.
%%%%%%%%%%%%%%%%%%%%%%
%
%
%
%
%
%
%%%%%%%%%%%%%%%%%%%%%%

%%%%%%%%%%%%%%%%%%%%%%
%%%%%%%%%%%%%%%%%%%%%%%%%%%%%%%%
%
%
%
%
%
%
%
%
%
%
%
%%%%%%%%%%%%%%%%%%%%%%%%%%%%%%%%

\begin{thebibliography}{99}
%%%%%
%
%%%%%
\bibitem{Arakawa1}
T. Arakawa,
\emph{
Generalized eta-functions and certain ray class invariants of real quadratic
fields.}
Math. Ann. \textbf{260} (1982), no. 4, 475--494.
%%%%%%
\bibitem{Arakawa2}
T.  Arakawa,
\emph{Dirichlet series 
$\sum_{n=1}^\infty\cot \pi \alpha\slash n^s$, 
Dedekind sums,
and Hecke $L$-functions for real quadratic fields.}
Comment. Math. Univ. St. Paul. \textbf{37} (1988), no. 2, 209--235.
%%%%%%
\bibitem{AIK}
T. Arakawa, T. Ibukiyama, M. Kaneko,
\emph{Bernoulli numbers and zeta functions.}
With an appendix by Don Zagier.
Springer Morogr. Math.
Springer, Tokyo, 2014.
xii+274 pp.
%%%%%
\bibitem{Berndt}
B. C. Berndt,
\emph{Dedekind sum and a paper of G. H. Hardy.}
J. London Math. Soc. (2) \textbf{13} (1976), no. 1, 129--137.
\bibitem{Lerch}
M. Lerch,
\emph{Sur une s\'erie analogue aux fonctions modulaires.}
C. R. Acad. Sci., Paris \textbf{138} (1904), 952--954.
%%%%%%%%%%%%%%%%%%%%%%
\end{thebibliography}
\end{document}